\documentclass[12pt,letterpaper]{article}

\usepackage{amsmath, amssymb, amsfonts, amsthm}

\usepackage{hyperref}
\usepackage{lineno}
\usepackage[dvipsnames]{xcolor}
\usepackage{tikz}

\hypersetup{
  colorlinks   = true,
  urlcolor     = blue, 
  linkcolor    = blue,
  citecolor   = red
}

\newtheorem{theorem}{Theorem}

\newtheorem{cor}[theorem]{Corollary}
\newtheorem{lemma}[theorem]{Lemma}

\newtheorem*{theorem*}{Theorem}
\newtheorem*{thm*}{Theorem}
\newtheorem*{lemma*}{Lemma}
\newtheorem*{conj*}{Conjecture}
\theoremstyle{definition}
\newtheorem{definition}[theorem]{Definition}

\numberwithin{theorem}{section}
\numberwithin{equation}{section}
\numberwithin{figure}{section}

\newcommand{\oeis}[1]{#1}

\newcommand{\perm}[1]{\mathcal{S}_{#1}}
\newcommand{\permpat}[2]{\perm{#1}(#2)}

\newcommand{\lehmer}[1]{\mathcal{L}_{#1}}
\newcommand{\lehmerpat}[2]{\lehmer{#1}(#2)}

\DeclareMathOperator{\inv}{inv}
\DeclareMathOperator{\Max}{Max}
\DeclareMathOperator{\supp}{supp}

\newcommand{\maxset}[2]{\Max \lehmerpat{#1}{#2}}

\newcommand{\positive}[1]{\supp(#1)}

\newcounter{x}
\newcounter{y}

\newcommand\emptycol[1]{
   \draw[thick] ( - #1, 0) -- (-#1, #1) -- (-#1-1, #1) -- (-#1-1, 0) -- cycle;   
   \foreach \c in {1,...,#1} {
      \draw[thick] (-#1, \c) -- (-#1-1, \c);
  }      
}

\newcommand\fillcol[3]{
    \fill[#3!50] (-#1, 0) -- (-#1, #2) -- (-#1-1, #2) -- (-#1-1, 0) -- cycle;
}

\newcommand\newtri[3][(0,0)]{
 \setcounter{y}{0}
  \foreach \a in {#2} {
      \addtocounter{y}{1}
    }

\begin{scope}[shift={#1}]
\begin{scope}[scale=0.33, shift={(-\value{y}/2,0)}]    
 \setcounter{x}{0}
  \foreach \a in {#2} {
      \addtocounter{x}{-1}
      \fillcol{\value{x}}{\a}{#3}
    }
  \end{scope}
\end{scope}

\begin{scope}[shift={#1}]
\begin{scope}[scale=0.33, shift={(\value{y}/2+1,0)}]    
 \setcounter{x}{0}
  \foreach \a in {#2} {
      \addtocounter{x}{1}
      \emptycol{\value{x}}
    }
  \end{scope}
\end{scope}
}

\title{Maximal Lehmer Codes for Permutations with Classical and Consecutive 321-Avoidance}

\author{Andrew Beveridge\footnote{Department of Mathematics, Statistics and Computer Science, Macalester College, Saint Paul, MN, USA 55105}, Yufan Hu$^*$, and Yucheng Liu$^*$}

\date{}

\begin{document}

\maketitle

\begin{abstract}
Let $\mathcal{S}_n(321)$ and $\mathcal{S}_n(\underline{321})$ denote the sets of $n$-permutations avoiding the classical pattern $321$ and the consecutive pattern $\underline{321}$, respectively. Permutations are in bijection with Lehmer codes, a type of inversion sequence. Using Lehmer codes, we create the corresponding weighted posets $\mathcal{L}_n(321)$ and $\mathcal{L}_n(\underline{321})$, where the weight of a code is the inversion number of its permutation. We show that there are $2^{n-2}$ maximal elements of  $\mathcal{L}_n(321)$, while the maximal elements of $\mathcal{L}_n(\underline{321})$ are enumerated by the Padovan numbers. We also show that when $n$ is even, each of these posets has a unique maximum weight element, and that when $n$ is odd, there are two maximum weight elements. These maximum weight Lehmer codes correspond to the pattern avoiding permutations with maximum inversion number.
\end{abstract}


\medskip

%\noindent
%{\bf OEIS Sequences:} 
%\oeis{A000034},
%\oeis{A000079}, 
%\oeis{A000108}, 
%\oeis{A000931}, 
%\oeis{A000982},
%\oeis{A002620}
%
%
%\medskip


\section{Introduction}

The study of pattern avoidance in permutations has grown into a central topic in enumerative combinatorics since the foundational work of Simion and Schmidt \cite{simion1985restricted}. 
Originally motivated by the study of stack-sortable permutations and algorithm analysis by Knuth \cite{knuth}, 
the field has expanded to include vincular patterns (also called generalized patterns), introduced by Babson and Steingr\'imsson \cite{Babson2000}. Unlike classical patterns where constraints are purely positional, vincular patterns require that certain adjacent entries in the pattern also be adjacent in the permutation.
 Comprehensive overviews of classical and vincular pattern avoidance can be found in Kitaev \cite{kitaev} and Steingr\'imsson \cite{steingrimsson}. 

Given permutations $\tau = \tau_1 \tau_2 \cdots \tau_k \in \perm{k}$ and $\pi = \pi_1 \pi_2 \cdots \pi_n \in \perm{n}$, we say that $\pi$ \emph{contains} $\tau$ when there exist indices $1 \leq i_1 < i_2 < \cdots < i_k \leq n$ such that the entries of subsequence $\pi_{i_1} \pi_{i_2} \cdots \pi_{i_k}$ are in the same relative order as the entries of $\tau$; otherwise, we say that $\pi$ \emph{avoids} $\tau$. 
We use $\permpat{n}{\tau}$ to denote the subset of $\tau$-avoiding permutations in $\perm{n}$.
 A \emph{vincular pattern} $\tau$ is a permutation in $\perm{k}$, some of whose consecutive entries are underlined. If $\pi \in \perm{n}$ contains vincular pattern $\tau$, and $\tau$ contains $\underline{\tau_i \tau_{i+1} \cdots \tau_j}$, then the entries of $\pi$ corresponding to $\tau_i \tau_{i+1} \cdots \tau_j$ in $\pi$ must be adjacent.

A common goal of pattern avoidance research is to enumerate avoidance classes or to establish bijections to other combinatorial families, for example, see Mansour and Shattuck \cite{mansour2022}, and Frosini, Guerrini, and Rinaldi \cite{Frosini2025}. Beveridge, Heysse, and Robertson \cite{BHR} study the family $\permpat{n}{\underline{32}1}$, 
which is enumerated by the Bell numbers $B_n$, sequence \oeis{A000110} in the OEIS \cite{oeis}. 
They translate these permutations into Lehmer codes in order to characterize the structure of these sets with respect to permutation inversions.

\begin{definition}
\label{def:lehmer}
    The collection $\lehmer{n}$ of \emph{Lehmer codes} of length $n$ is 
    $$
    \lehmer{n} = \{ (p_1, p_2, \ldots, p_n) : 0 \leq p_i \leq n-i \text{ for } 1 \leq i \leq n \}.
    $$
    For a permutation $\pi \in \perm{n}$, its Lehmer code is
$$
L(\pi) := ( p_1, p_2,\ldots , p_n) \quad \mbox{where} \quad
p_i = \left| \left\{ j > i : \pi_j < \pi_i \right\} \right|,
$$     
and this mapping $L : \perm{n} \rightarrow \lehmer{n}$  is a bijection \cite{lehmer}.
Given a permutation pattern $\tau$, we define
$$
\lehmerpat{n}{\tau} = \{ L(\pi) : \pi \in \permpat{n}{\tau}\}
$$
to be the collection of codes for the $\tau$-avoiding permutations.
The \emph{weight} of a Lehmer code is $w(p) = \sum_{i=1}^{n} p_i$, and it is clear that the inversion number of $\pi$ is $\inv \pi = w(p)$.
\end{definition}

Lehmer codes have a natural poset structure $(\lehmer{n}, \preceq)$, as studied by Denoncourt \cite{denoncourt}, Tomie \cite{tomie}, and Bouvel, Ferrari, and Tenner \cite{bouvel}. Given $p, q \in \lehmer{n}$, we have $p \preceq q$ when $p_i \leq q_i$ for $1 \leq i \leq n$. For pattern $\tau$, we define $(\lehmerpat{n}{\tau}, \preceq)$ to be the induced subposet on
$\lehmerpat{n}{\tau}$.
We denote its maximal elements by $\maxset{n}{\tau}$ and its maximum weight elements by $\Max_w\lehmerpat{n}{\tau}$. Note that the maximum weight elements correspond to $\tau$-avoiding permutations with maximum inversion number. The sets $\maxset{n}{\underline{32}1}$ and $\Max_w \lehmerpat{n}{\underline{32}1}$ were characterized by Beveridge, Heysse and Robertson \cite{BHR}. They show that $\maxset{n}{\underline{32}1}$ is enumerated by the Fibonacci numbers (OEIS sequence \oeis{A000045}) and $\Max_w \lehmerpat{n}{\underline{32}1}$ is enumerated by sequence \oeis{A209561} in the OEIS \cite{oeis}. Meanwhile, Beveridge, Hu, and Liu \cite{BHL2026} showed that the reverse-complement map for permutations induces a natural bijection from $\maxset{n}{\underline{32}1}$ to $\maxset{n}{3\underline{21}}$.

We continue the study of Lehmer codes whose permutations have $321$-type pattern avoidance. We investigate the two remaining cases: the Lehmer posets $\lehmerpat{n}{321}$ and $\lehmerpat{n}{\underline{321}}$, corresponding to the classical and consecutive patterns. It is well-known that $|\permpat{n}{321}|=C_n$, the $n$th Catalan number, 
sequence \oeis{A000108} in the OEIS \cite{oeis}.
Meanwhile
$|\permpat{n}{\underline{321}}|$ is enumerated by OEIS sequence \oeis{A049774}.   
This sequence is commonly described as the number of permutations avoiding the consecutive pattern $\underline{123}$. By applying the complement map $\pi_i \mapsto n+1-\pi_i$, $\underline{123}$-avoidance is equivalent to $\underline{321}$-avoidance. Consecutive pattern avoidance, including this family, was studied
systematically by Elizalde and Noy \cite{ElizaldeNoy2003}; for a survey of the topic, see Elizalde \cite{Elizalde2016}.
We characterize the maximal Lehmer codes and the maximum weight Lehmer codes for $\lehmerpat{n}{321}$ and $\lehmerpat{n}{\underline{321}}$. Figure \ref{fig:posets-4} shows the posets $\lehmerpat{4}{321}$ and $\lehmerpat{4}{\underline{321}}$ where we display $p \in \lehmer{n}$ as columns of (filled) boxes in a triangular grid of size $n-1$, where column $i$ contains $0 \leq p_i \leq n-i$ boxes.

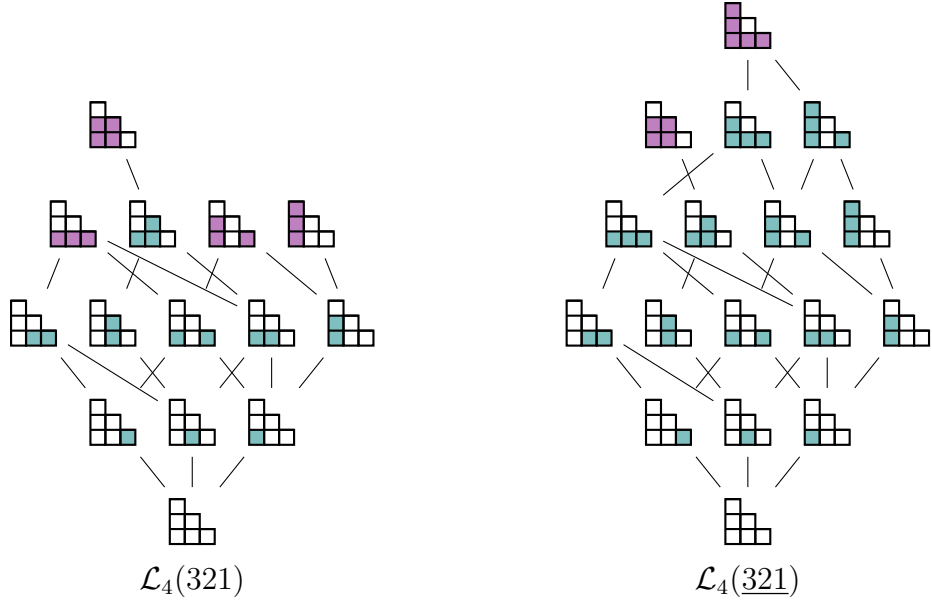
\begin{figure}[t]
\centering
\begin{tikzpicture}[scale=0.525]

\begin{scope}

\node (000) at (0,0) { \begin{tikzpicture}[scale=0.6]\newtri[(-1,0)]{0,0,0}{teal} \end{tikzpicture}};

\node (001) at (-2,2.5) {\begin{tikzpicture}[scale=0.6]\newtri[(-1,0)]{0,0,1}{teal}\end{tikzpicture}};
\node (010) at (0,2.5) {\begin{tikzpicture}[scale=0.6]\newtri[(-1,0)]{0,1,0}{teal}\end{tikzpicture}};
\node (100) at (2,2.5) {\begin{tikzpicture}[scale=0.6]\newtri[(-1,0)]{1,0,0}{teal}\end{tikzpicture}};

\node (011) at (-4,5) {\begin{tikzpicture}[scale=0.6]\newtri[(-1,0)]{0,1,1}{teal}\end{tikzpicture}};
\node (020) at (-2,5) {\begin{tikzpicture}[scale=0.6]\newtri[(-1,0)]{0,2,0}{teal}\end{tikzpicture}};
\node (101) at (0,5) {\begin{tikzpicture}[scale=0.6]\newtri[(-1,0)]{1,0,1}{teal}\end{tikzpicture}};
\node (110) at (2,5) {\begin{tikzpicture}[scale=0.6]\newtri[(-1,0)]{1,1,0}{teal}\end{tikzpicture}};
\node (200) at (4,5) {\begin{tikzpicture}[scale=0.6]\newtri[(-1,0)]{2,0,0}{teal}\end{tikzpicture}};

\node (111) at (-3,7.5) {\begin{tikzpicture}[scale=0.6]\newtri[(-1,0)]{1,1,1}{violet}\end{tikzpicture}};
\node (120) at (-1,7.5) {\begin{tikzpicture}[scale=0.6]\newtri[(-1,0)]{1,2,0}{teal}\end{tikzpicture}};
\node (201) at (1,7.5) {\begin{tikzpicture}[scale=0.6]\newtri[(-1,0)]{2,0,1}{violet}\end{tikzpicture}};
\node (300) at (3,7.5) {\begin{tikzpicture}[scale=0.6]\newtri[(-1,0)]{3,0,0}{violet}\end{tikzpicture}};

\node (220) at (-2,10) {\begin{tikzpicture}[scale=0.6]\newtri[(-1,0)]{2,2,0}{violet}\end{tikzpicture}};

\foreach \from/\to in {
000/001, 000/010, 000/100,
001/011, 001/101,
010/011, 010/020, 010/110,
100/101, 100/110, 100/200,
011/111,
020/120,
101/111, 101/201,
110/111, 110/120,
200/201, 200/300,
120/220}
\draw (\from)--(\to);

\node at (0,-1.5) {$\lehmerpat{4}{321}$};

\end{scope}

\begin{scope}[shift={(14,0)}]

\node (000) at (0,0) { \begin{tikzpicture}[scale=0.6]\newtri[(-1,0)]{0,0,0}{teal} \end{tikzpicture}};

\node (001) at (-2,2.5) {\begin{tikzpicture}[scale=0.6]\newtri[(-1,0)]{0,0,1}{teal}\end{tikzpicture}};
\node (010) at (0,2.5) {\begin{tikzpicture}[scale=0.6]\newtri[(-1,0)]{0,1,0}{teal}\end{tikzpicture}};
\node (100) at (2,2.5) {\begin{tikzpicture}[scale=0.6]\newtri[(-1,0)]{1,0,0}{teal}\end{tikzpicture}};

\node (011) at (-4,5) {\begin{tikzpicture}[scale=0.6]\newtri[(-1,0)]{0,1,1}{teal}\end{tikzpicture}};
\node (020) at (-2,5) {\begin{tikzpicture}[scale=0.6]\newtri[(-1,0)]{0,2,0}{teal}\end{tikzpicture}};
\node (101) at (0,5) {\begin{tikzpicture}[scale=0.6]\newtri[(-1,0)]{1,0,1}{teal}\end{tikzpicture}};
\node (110) at (2,5) {\begin{tikzpicture}[scale=0.6]\newtri[(-1,0)]{1,1,0}{teal}\end{tikzpicture}};
\node (200) at (4,5) {\begin{tikzpicture}[scale=0.6]\newtri[(-1,0)]{2,0,0}{teal}\end{tikzpicture}};

\node (111) at (-3,7.5) {\begin{tikzpicture}[scale=0.6]\newtri[(-1,0)]{1,1,1}{teal}\end{tikzpicture}};
\node (120) at (-1,7.5) {\begin{tikzpicture}[scale=0.6]\newtri[(-1,0)]{1,2,0}{teal}\end{tikzpicture}};
\node (201) at (1,7.5) {\begin{tikzpicture}[scale=0.6]\newtri[(-1,0)]{2,0,1}{teal}\end{tikzpicture}};
\node (300) at (3,7.5) {\begin{tikzpicture}[scale=0.6]\newtri[(-1,0)]{3,0,0}{teal}\end{tikzpicture}};

\node (211) at (0,10) {\begin{tikzpicture}[scale=0.6]\newtri[(-1,0)]{2,1,1}{teal}\end{tikzpicture}};
\node (220) at (-2,10) {\begin{tikzpicture}[scale=0.6]\newtri[(-1,0)]{2,2,0}{violet}\end{tikzpicture}};
\node (301) at (2,10) {\begin{tikzpicture}[scale=0.6]\newtri[(-1,0)]{3,0,1}{teal}\end{tikzpicture}};

\node (311) at (0,12.5) {\begin{tikzpicture}[scale=0.6]\newtri[(-1,0)]{3,1,1}{violet}\end{tikzpicture}};

\foreach \from/\to in {
000/001, 000/010, 000/100,
001/011, 001/101,
010/011, 010/020, 010/110,
100/101, 100/110, 100/200,
011/111,
020/120,
101/111, 101/201,
110/111, 110/120,
200/201, 200/300,
111/211,
120/220,
201/211, 201/301,
300/301,
211/311,
301/311}
\draw (\from)--(\to);

\node at (0,-1.5) {$\lehmerpat{4}{\underline{321}}$};

\end{scope}

\end{tikzpicture}
\caption{The posets of $321$-avoiding and $\underline{321}$-avoiding Lehmer codes of length $4$. The maximal Lehmer codes are colored violet.}
\label{fig:posets-4}
\end{figure}

For classical $321$-avoidance, we show that a maximal code $p \in \Max \lehmerpat{n}{321}$ is determined by its \emph{zero set}
$$
Z(p) = \{ i \in [n] : p_i=0 \}.
$$
This gives the following enumeration by sequence \oeis{A000079} in the OEIS \cite{oeis}.

\begin{theorem}
\label{thm:number_maximal_321}
For $n\geq 2$, we have
$
|\maxset{n}{321}|=2^{n-2}.
$
\end{theorem}
Meanwhile, there are either one or two maximum weight Lehmer codes, depending on the parity of $n$ (sequence \oeis{A000034} in the OEIS \cite{oeis}). These correspond to 321-avoiding permutations with maximum inversion number $\lfloor n^2 /4 \rfloor$, given by sequence \oeis{A002620} in the OEIS \cite{oeis}.

\begin{theorem}
\label{thm:maximum_weight_321}
For $n \geq 2$, we have
$$
\max \{ \inv(\pi) : \pi \in \permpat{n}{321}  \} = 
\max \{ w(p) : p \in \lehmerpat{n}{321} \}
=
\left\lfloor \frac{n^2}{4} \right\rfloor.
$$
Moreover, if $n=2r$ is even, the unique maximum weight Lehmer code is
\[
p=(\underbrace{r,\ldots,r}_{r}, \underbrace{0,\ldots,0}_{r}),
\]
while if $n=2r+1$ is odd, the maximum weight Lehmer codes are
\[
p=(\underbrace{r,\ldots,r}_{r+1}, \underbrace{0,\ldots,0}_{r}) \qquad \text{ and } \qquad p =  (\underbrace{r+1,\ldots,r+1}_{r}, \underbrace{0,\ldots,0}_{r+1}).
\]
\end{theorem}

For example, the unique maximum weight code in $\lehmerpat{8}{321}$ is
$(4,4,4,4,0,0,0,0)$, which corresponds to the permutation $56781234$. 
Meanwhile, the maximum weight codes in $\lehmerpat{9}{321}$ are
$(4,4,4,4,4,0,0,0,0)$ and $(5,5,5,5,0,0,0,0,0)$, corresponding to the permutations $567891234$ and $678912345$, respectively.

Turning to the consecutive pattern $\underline{321}$, we show that a
maximal code $p \in \Max \lehmerpat{n}{\underline{321}}$ is determined by its \emph{descent set}
$$
D(p) = \{ i \in [n-1] : p_i > p_{i+1} \}.
$$
Furthermore, the collection of descent sets for $\Max \lehmerpat{n}{\underline{321}}$ is precisely the collection of inclusion-maximal subsets of $[n-1]$ with no consecutive elements. This leads to a novel appearance of the Padovan numbers, sequence \oeis{A000931} in the OEIS \cite{oeis}.

\begin{theorem}
\label{thm:consec_321_number}
The number of maximal $\underline{321}$-avoiding Lehmer codes 
is given by the Padovan number
$
|\maxset{n}{\underline{321}}| = P_{n-1},
$
where $P_0=P_1=1$, $P_2=2$, and
$P_n = P_{n-2} + P_{n-3}$ for $n \geq 3$.
\end{theorem}

Once again, there are either one or two maximum weight Lehmer codes, depending on the parity of $n$. These correspond to $\underline{321}$-avoiding permutations with maximum inversion number $\left\lceil (n-1)^2/2\right\rceil$, a shift of sequence \oeis{A000982} in the OEIS \cite{oeis}.

\begin{theorem}
\label{thm:maximum-weight-321}
For $n \geq 2$, we have
$$
\max \{ \inv(\pi) : \pi \in \permpat{n}{\underline{321}}  \} = 
\max \{ w(p) : p \in \lehmerpat{n}{\underline{321}} \}
=
\left\lceil \frac{(n-1)^2}{2}\right\rceil.
$$
Moreover, if $n = 2r$ is even, the unique maximum weight Lehmer code is
$$
(2r-1, 2r-3, 2r-3, 2r-5, 2r-5, \ldots, 3, 3, 1, 1, 0)
$$
and when $n=2r+1$ is odd, the maximum weight Lehmer codes are
$$
(2r, 2r-2, 2r-2, 2r-4, 2r-4, \ldots, 4, 4, 2, 2, 0, 0)
$$
and
$$
(2r-1, 2r-1, 2r-3, 2r-3, \ldots, 3, 3, 1, 1,  0).
$$
\end{theorem}

For example, the unique maximum weight code in $\lehmerpat{8}{\underline{321}}$ is
$(7,5,5,3,3,1,1,0)$, which corresponds to the permutation $86745231$. 
Meanwhile, the maximum weight codes in $\lehmerpat{9}{\underline{321}}$ are
$(8,6,6,4,4,2,2,0,0)$ and $(7,7,5,5,3,3,1,1,0)$, corresponding to the permutations $978563412$ and $896745231$, respectively.

\subsection{Roadmap and Notation}

The paper is structured as follows. Section \ref{sec:prelim} describes how to map between permutations and Lehmer codes.  Section \ref{sec:classic-321}
contains our results about classical $321$-avoiding permutations.
Section \ref{sec:consec-321}
contains our results about consecutive $\underline{321}$-avoiding permutations.

In addition to  the zero set $Z(p)$ and the descent set $D(p)$, we will also use  the \emph{support set}
$$
\positive{p} = \{ i \in [n-1] : p_i > 0 \}.
$$
Furthermore, we will always list the elements of these sets in increasing order, with the following naming conventions:
$$
\begin{array}{rcll}
    Z(p) &=& \{ z_1, z_2, \ldots, z_k \} &  \mbox{where } z_1 < z_2 < \cdots < z_k, \\
    \positive{p} &=& \{s_1, s_2, \ldots, s_\ell \} &  \mbox{where } s_1 < s_2 < \cdots < s_\ell, \\
    D(p) &=& \{d_1, d_2, \ldots, d_m \} &  \mbox{where } d_1 < d_2 < \cdots < d_m.
\end{array}
$$
Finally, we assume $n \geq 2$ throughout the paper, to avoid trivial arguments about $\perm{1}$.

\section{Preliminaries}

\label{sec:prelim}

We describe the bijections between permutations $\perm{n}$ and Lehmer codes $\lehmer{n}$. The decoding process, which transforms code $p$ into permutation $\pi$, will be used repeatedly in the sections that follow. 

Starting with $\pi = \pi_1 \pi_2 \cdots \pi_n \in \perm{n}$, its Lehmer code $p = (p_1, p_2, \ldots, p_n)$ is given by $p_i = | \{ j > i : \pi_j < \pi_i \}|$. For example, if $\pi = 2431$ then $p=(1,2,1,0)$. Indeed, there is one entry to the right of $\pi_1=2$ that is smaller, there are two entries to the right of $\pi_2=4$ that are smaller, and there is one entry to the right of $\pi_3=3$ that is smaller.

We now consider the inverse mapping: decoding $p \in \lehmer{n}$ to recover its  corresponding permutation $\pi \in \perm{n}$.
First, create a candidate list $C=(1,2,\ldots, n)$. We  process the entries of $p=(p_1, \ldots, p_n)$ from left to right, and we update $C$ as we go. At step $i$, we skip over the $p_i$ smallest elements of $C$, setting $\pi_i$ to be the entry in position $(p_i+1)$ of $C$, and we remove this value from the list $C$. This guarantees that there will be exactly $p_i$ elements smaller than $\pi_i$ among the entries $\pi_{i+1}, \ldots, \pi_n$.

For example, suppose that we want to decode $p=(3,1,0,0) \in \lehmer{4}$. We start with $C=(1,2,3,4)$. We set $\pi_1=4$ (skipping over $p_1=3$ entries of $C$), and update $C=(1,2,3)$. Next, we set $\pi_2=2$ (skipping over $p_2=1$ entries of $C$), and update $C=(1,3)$. We then set $\pi_3=1$ (skipping over $p_3=0$ entries of $C$), and update $C=(3)$. Finally, we set $\pi_4=3$. So the corresponding permutation is $\pi=4213 \in \perm{4}$.

\section{Lehmer Codes for Classical $321$-Avoiding Permutations}

\label{sec:classic-321}

In this section, we turn our attention to classical $321$-avoiding permutations. The main goal is to describe the maximal elements of $\lehmerpat{n}{321}$ in terms of their zero sets.

We begin by showing  that $321$-avoidance is equivalent to a strict increase condition on the values $p_i+i$ on the support of $p$.
This leads to a sharp upper bound on each positive entry, expressed as the number of zeros to its right. Maximality will force these bounds to be attained, giving a complete zero set description of the maximal codes.

\begin{lemma}
 \label{lem:321-postive-increase}   
Let $p=(p_1,\ldots,p_n)\in \lehmer{n}$ with
$\positive{p} = \{s_1,s_2,\ldots,s_\ell\}$. Then $p\in \lehmerpat{n}{321}$ if and only if
\begin{equation}
\label{eqn:321-postive-increase}
p_{s_1}+s_1<p_{s_2}+s_2<\cdots<p_{s_\ell}+s_{\ell}.
\end{equation}
\end{lemma}

\begin{proof}
It is well-known that a permutation is 321-avoiding if and only if its entries can be partitioned into (at most) two increasing subsequences. We claim that a valid partitioning of the indices is to use the sets $Z(p) = \{ z_1, z_2, \ldots, z_k \}$ and $\positive{p} = \{ s_1, s_2, \ldots, s_{\ell} \}$. Indeed, after decoding $\pi$ from its Lehmer code $p$ (as described in Section \ref{sec:prelim}), the subsequence $\pi_{z_1}, \pi_{z_2}, \ldots, \pi_{z_k}$ is clearly increasing.
Now suppose that $\pi_{s_1}, \pi_{s_2}, \ldots, \pi_{s_{\ell}}$ is not increasing, so there exists $1 \leq i \leq \ell-1$ such that $\pi_{s_i} > \pi_{s_{i+1}}$. Let $z_{j} = \min \{ z_h \in Z(p) : z_h > s_{i+1} \}$ be the smallest zero set index larger than index $s_{i+1}$. (Index $z_{j}$ exists because $z_{k}=n > s_{i+1}$, so the set under consideration is nonempty.) Then $\pi_{s_{i}} > \pi_{s_{i+1}} > \pi_{z_{j}}$, so $\pi$ contains a 321-pattern, a contradiction.

We now prove that $\pi_{s_1} < \pi_{s_2} < \cdots <  \pi_{s_{\ell}}$ if and only if condition \eqref{eqn:321-postive-increase} holds. Equivalently, we show that for $1 \leq i \leq \ell-1$, we have $\pi_{s_i} < \pi_{s_{i+1}}$ if and only if  $p_{s_i} + s_i < p_{s_{i+1}} + s_{i+1}$.
Suppose that we are decoding $p$, and we have reached index $s_i>0$, with current candidate list $C$. We set $\pi_{s_i}$ to be the $(p_{s_i}+1)$th entry of list $C$, and remove this value from $C$. Next, we encounter $j=s_{i+1}-s_i-1$ indices in $Z(p)$, which account for the smallest $j$ entries of $C$, which we remove after their assignment. Now we have reached index $s_{i+1}$, and we assign value $\pi_{s_{i+1}} > \pi_{s_i}$ if and only if $p_{s_{i+1}} \geq p_{s_i} -j$. The latter inequality is equivalent to $p_{s_{i+1}} + s_{i+1} \geq p_{s_i} + s_i +1$, which is the same as $p_{s_{i+1}} + s_{i+1} > p_{s_i} + s_i$.
\end{proof}

For example, consider $p=(0,2,2,0,0,1,0,1,0)\in \lehmer{9}$. We have $\positive{p}=\{2,3,6,8\}$, and these positions satisfy $p_2+2=4 <  p_3+3=5 < p_6+6=7 < p_8+8 = 9.$ By Lemma \ref{lem:321-postive-increase}, the corresponding permutation $\pi$ is 321-avoiding. Indeed, we can decode $p$ 
to find that $\pi = 145237698$. Note that partitioning $\pi$ into subsequences on $Z(p)$ and $\supp(p)$ yields increasing sequences $12368$ and $4579$, as guaranteed in the proof of the lemma.

The increasing condition \eqref{eqn:321-postive-increase} gives more than a test for $321$-avoidance: it also controls how large each positive entry can be. In particular, the size of a positive entry is bounded by the number of zeros to its right.

%%%%%%%%%%%%%%%%%%%% Maximal %%%%%%%%%%%%%%%%%%%%%

\begin{lemma}
\label{lem:321-zero-bound}
Let $p\in \lehmerpat{n}{321}$ with $\positive{p}=\{s_1,s_2,\ldots,s_{\ell}\}.$ Then for every $1\le j\le \ell$,
\[
p_{s_j}\le |\{z\in Z(p):s_j < z \leq n \}|=n-\ell+j-s_j.
\]
\end{lemma}

\begin{proof}
First, we note that for any Lehmer code, we have $p_i + i \leq (n-i)+i = n$ for $1 \leq i \leq n$. Since $p\in \lehmerpat{n}{321}$, Lemma \ref{lem:321-postive-increase} implies that $p_{s_1}+s_1,\ p_{s_2}+s_2,\ \ldots,\ p_{s_{\ell}}+s_{\ell}$ form a strictly increasing sequence of $\ell$ positive integers, all at most $n$. Hence the $j$-th value is at most the $j$-th largest possible value in such a sequence, that is $p_{s_j}+s_j\le n-(\ell-j)$.

It remains to show that $n-\ell+j-s_j$ is the number of zeros strictly to the right of $s_j$. There are $n-s_j$ positions to the right of $s_j$. Of these, exactly $\ell-j$ are positive positions, namely $s_{j+1},s_{j+2}, \ldots,s_{\ell}$. Therefore the number of zero positions strictly to the right of $s_j$ is
\[
(n-s_j)-(\ell-j)=n-\ell+j-s_j.
\]
\end{proof}

For example, for $p=(3,0,2,0,1,0)\in \lehmerpat{6}{321}$, we have $Z(p)=\{2,4,6\}$. Since $\positive{p}=\{1,3,5\}$, Lemma \ref{lem:321-zero-bound} gives
\[
p_1\le |\{2,4,6\}|=3,\qquad p_3\le |\{4,6\}|=2,\qquad p_5\le |\{6\}|=1.
\]
In this case, all three bounds are attained with equality. This example suggests the role of maximality: once the zero set is fixed, a maximal code should attain all of the bounds from Lemma \ref{lem:321-zero-bound}. This gives the following characterization.

\begin{theorem}
\label{thm:maximal_condition_321}
Let $n\geq 2$ and let $p=(p_1,\ldots,p_n)\in \lehmerpat{n}{321}$ with $\positive{p}=\{s_1,\ldots,s_{\ell}\}$. Then $p\in \maxset{n}{321}$ if and only if $s_1=1$ and
\[
p_{s_j}=|\{z\in Z(p):s_j<z\leq n\}| \mbox{ for } 1 \leq j \leq \ell.
\]
That is, every positive entry $p_{s_j}$ is equal to the number of zeros strictly to its right.
\end{theorem}

Figure \ref{fig:max321} shows the codes in $\Max \lehmerpat{5}{321}$ with their zero sets, illustrating the main point of Theorem \ref{thm:maximal_condition_321}. Once the zero set is fixed, every positive entry is forced to count the zeros to its right. In other words, the zero set  actually determines the code.

\begin{figure}[t]
\begin{center}
\begin{tikzpicture}[scale=.6]

    \newtri[(0,0)]{1,1,1,1}{teal};
    \node at (0,-0.8) {\footnotesize $\{5\}$};

    \newtri[(3,0)]{2,0,1,1}{teal};
    \node at (3,-0.8) {\footnotesize $\{2,5\}$};

    \newtri[(6,0)]{2,2,0,1}{teal};
    \node at (6,-0.8) {\footnotesize $\{3,5\}$};

    \newtri[(9,0)]{2,2,2,0}{teal};
    \node at (9,-0.8) {\footnotesize $\{4,5\}$};

    \newtri[(12,0)]{3,0,0,1}{teal};
    \node at (12,-0.8) {\footnotesize $\{2,3,5\}$};

    \newtri[(15.0,0)]{3,0,2,0}{teal};
    \node at (15.0,-0.8) {\footnotesize $\{2,4,5\}$};

    \newtri[(18,0)]{3,3,0,0}{teal};
    \node at (18,-0.8) {\footnotesize $\{3,4,5\}$};

    \newtri[(21,0)]{4,0,0,0}{teal};
    \node at (21,-0.8) {\footnotesize $\{2,3,4,5\}$};

    \node at (-2, 0.5) {\footnotesize $p$};
    \node at (-2, -0.8) {\footnotesize $Z(p)$};

\end{tikzpicture}
\end{center}
\caption{The triangle representations of the  maximal Lehmer codes in $\maxset{5}{321}$, labeled by their zero sets.}
\label{fig:max321}
\end{figure}
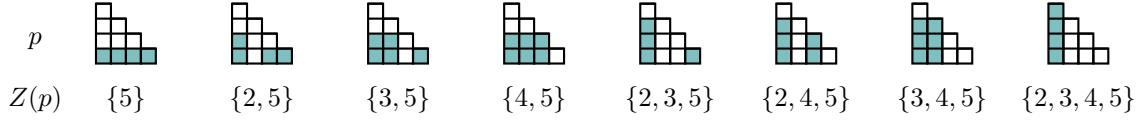

\begin{proof}
First suppose that $p\in \maxset{n}{321}$. If $\positive{p}=\emptyset$, then $p=(0,\ldots,0)$ is clearly not a maximal Lehmer code. Thus $\positive{p}\neq\emptyset$, so $s_1$ is defined. 

If $p_1=0$, then increasing $p_1$ from $0$ to $1$ gives another valid Lehmer code $p'$. We then have $p'_1+1=2<p_{s_1}+s_1=p'_{s_1}+s_1$. By Lemma \ref{lem:321-postive-increase}, the code $p'$ is still $321$-avoiding. This contradicts the maximality of $p$. Therefore $p_1>0$ which means that $s_1=1$.

Now let $s_j\in \positive{p}$. By Lemma \ref{lem:321-zero-bound}, we have $p_{s_j}\le |\{z\in Z(p):z>s_j\}|$. We show that equality holds. Suppose not, and choose the largest $j$ such that $p_{s_j}<|\{z\in Z(p):z>s_j\}|$. By Lemma \ref{lem:321-zero-bound}, this is equivalent to $p_{s_j}+s_j<n-\ell+j$. Increase $p_{s_j}$ by $1$, leaving all other entries unchanged, to obtain a code $p'$. Since
\[
p'_{s_j}+s_j=p_{s_j}+1+s_j\le n-\ell+j\le n,
\]
the code $p'$ is still a valid Lehmer code. Clearly $\positive{p'}=\positive{p}$. If $j>1$, then
\[
p'_{s_j}+s_j>p_{s_j}+s_j>p_{s_{j-1}}+s_{j-1} =p'_{s_{j-1}}+s_{j-1}
\]
by Lemma \ref{lem:321-postive-increase}. If $j<\ell$, then the maximal choice of $j$ gives
\[
p_{s_{j+1}}+s_{j+1}=n-\ell+j+1,
\]
so
\[
p'_{s_j}+s_j=p_{s_j}+1+s_j\le n-\ell+j < p_{s_{j+1}}+s_{j+1}=p'_{s_{j+1}}+s_{j+1}.
\]
Thus the values
\[
p'_{s_1}+s_1,\ p'_{s_2}+s_2,\ \ldots,\ p'_{s_{\ell}}+s_{\ell}
\]
are still strictly increasing. By Lemma \ref{lem:321-postive-increase}, code $p'$ is $321$-avoiding, contradicting the maximality of $p$. Hence $p_{s_j}=|\{z\in Z(p):z>s_j\}|$ for all $1\le j\le \ell$.

Conversely, suppose that $p\in \lehmerpat{n}{321}$, $s_1=1$, and $p_{s_j}=|\{z\in Z(p):s_j<z\le n\}|$ for $1\le j\le \ell$. By Lemma \ref{lem:321-zero-bound}, this is equivalent to $p_{s_j}+s_j=n-\ell+j$ for $1\le j\le \ell$.

To prove maximality, let $q\in \lehmerpat{n}{321}$ satisfy $p\preceq q$, and write $\positive{q}=\{t_1,t_2,\cdots,t_m\}$ where $t_1 < t_2 < \cdots < t_m$. Since $p\preceq q$, we have $\positive{p}\subseteq \positive{q}$, so $m\ge \ell$. Also $s_1=1$, so $t_1=1$. As in the proof of Lemma \ref{lem:321-zero-bound}, the values
\[
q_{t_1}+t_1,\ q_{t_2}+t_2,\ldots, q_{t_m}+t_m
\]
form a strictly increasing sequence, all at most $n$. Hence
\[
q_{t_i}+t_i\le n-m+i
\]
for $1\le i \le m$.

If $m>\ell$, then $q_1+1\le n-m+1<n-\ell+1=p_1+1$, contradicting $q_1\ge p_1$. Thus $m=\ell$ and
\[
\positive{q}=\positive{p}=\{s_1,\ldots,s_{\ell}\}.
\]
Therefore, for every $1\le j\le \ell$,
\[
q_{s_j}+s_j\le n-\ell+j=p_{s_j}+s_j,
\]
so $q_{s_j}\le p_{s_j}$. Since $p\preceq q$, equality holds on $\positive{p}$. Hence $q=p$, which means that $p\in \maxset{n}{321}$.
\end{proof}

\begin{cor}
\label{cor:maximal_condition_321}
Let $Z \subset [n]$ such that $1 \notin Z$ and $n \in Z$. Then there exists a unique $p \in \maxset{n}{321}$ such that $Z(p)=Z$.
\end{cor}

\begin{proof}
Write $\overline{Z} = [n]\backslash Z=\{s_1, s_2, \ldots,s_{\ell}\}$ where $s_1< s_2 < \cdots<s_{\ell}$. Define $p=(p_1,\ldots,p_n)$ by setting $p_i=0$ for $i\in Z$ and $p_{s_j}=|\{z\in Z:s_j<z\leq n\}|$ for $1\leq j\leq \ell$; observe that $\positive{p}=\overline{Z}$. Furthermore, since $1 \in \overline{Z}$ and $n\in Z$, we have $s_1=1$. Also $p_{s_j}=n-\ell+j-s_j$, so $p_{s_j}+s_j=n-\ell+j$. Hence $p\in\lehmerpat{n}{321}$ by Lemma \ref{lem:321-postive-increase}, and $p\in\maxset{n}{321}$ by Theorem \ref{thm:maximal_condition_321}. Uniqueness follows from the same theorem.
\end{proof}

For example, let $Z=\{3,5,6,8\}\subseteq [8]$. By Corollary \ref{cor:maximal_condition_321}, the corresponding $p \in \maxset{8}{321}$ is $p=(4,4,0,3,0,0,1,0)$. This is the Lehmer code for the 321-avoiding permutation $\pi=56172384$. 
The example reflects the general situation. Once we require $1\notin Z$ and $n\in Z$, there are no further restrictions on the zero set. Thus counting maximal $321$-avoiding Lehmer codes is the same as counting these possible zero sets.

\begin{theorem*}[Theorem \ref{thm:number_maximal_321}]
For $n\geq 2$, we have $|\maxset{n}{321}|=2^{n-2}$.
\end{theorem*}

\begin{proof}
By Theorem \ref{thm:maximal_condition_321} and Corollary \ref{cor:maximal_condition_321}, $\maxset{n}{321}$ is in bijection with $\{ Z \subset [n]: 1 \notin Z \mbox{ and } n \in Z\}$. There are $2^{n-2}$ such subsets.
\end{proof}

We now use the zero set description to determine the maximum weight of a 321-avoiding Lehmer code. 

\begin{theorem*}[Theorem \ref{thm:maximum_weight_321}]
For $n \geq 2$, we have
$$
\max \{ \inv(\pi) : \pi \in \permpat{n}{321}  \} = 
\max \{ w(p) : p \in \lehmerpat{n}{321} \}
=
\left\lfloor \frac{n^2}{4} \right\rfloor.
$$
Moreover, if $n=2r$ is even, the unique maximum weight element is
\[
p=(\underbrace{r,\ldots,r}_{r}, \underbrace{0,\ldots,0}_{r}),
\]
while if $n=2r+1$ is odd, the maximum weight elements are
\[
p=(\underbrace{r,\ldots,r}_{r+1}, \underbrace{0,\ldots,0}_{r}) \qquad \text{ and } \qquad p =  (\underbrace{r+1,\ldots,r+1}_{r}, \underbrace{0,\ldots,0}_{r+1}).
\]
\end{theorem*}

\begin{figure}[t]
\begin{center}
\begin{tikzpicture}[scale=.6]

    % n = 2
    \newtri[(0,0)]{1}{teal};
    \node at (0,-0.8) {\footnotesize $(1,0)$};
    \node at (0,-1.55) {\footnotesize $n=2$};

    % n = 3
    \newtri[(3.2,0)]{1,1}{teal};
    \node at (3.2,-0.8) {\footnotesize $(1,1,0)$};

    \newtri[(5.8,0)]{2,0}{teal};
    \node at (5.8,-0.8) {\footnotesize $(2,0,0)$};

    \node at (4.5,-1.55) {\footnotesize $n=3$};

    % n = 4
    \newtri[(9.2,0)]{2,2,0}{teal};
    \node at (9.2,-0.8) {\footnotesize $(2,2,0,0)$};
    \node at (9.2,-1.55) {\footnotesize $n=4$};

    % n = 5
    \newtri[(13.0,0)]{2,2,2,0}{teal};
    \node at (13.0,-0.8) {\footnotesize $(2,2,2,0,0)$};

    \newtri[(16.2,0)]{3,3,0,0}{teal};
    \node at (16.2,-0.8) {\footnotesize $(3,3,0,0,0)$};

    \node at (14.6,-1.55) {\footnotesize $n=5$};

    % n = 6
    \newtri[(20.4,0)]{3,3,3,0,0}{teal};
    \node at (20.4,-0.8) {\footnotesize $(3,3,3,0,0,0)$};
    \node at (20.4,-1.55) {\footnotesize $n=6$};

\end{tikzpicture}
\end{center}
\caption{The triangle representations of the maximum weight $321$-avoiding Lehmer codes for $2\leq n\leq 6$. For odd $n$, there are two maximum weight elements. For even $n$, the maximum weight element is unique.}
\label{fig:maxweight321examples}
\end{figure}
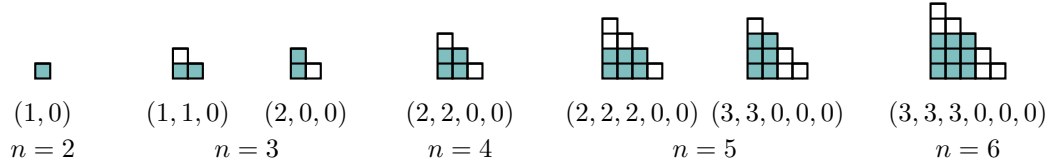

Figure \ref{fig:maxweight321examples} visualizes maximum weight Lehmer codes using our triangular representation. When $n=2r$ is even, the nonzero entries form an $r \times r$ square. When $n=2r+1$ is odd, the nonzero entries form either an $(r+1) \times r$ rectangle or an $r \times (r+1)$ rectangle.  

\begin{proof}
When $\pi$ has Lehmer code $p$, we have $\inv(\pi)=w(p)$, as noted in Definition \ref{def:lehmer}. This explains the first equality.

By Theorem \ref{thm:maximal_condition_321}, every maximal element $p\in \maxset{n}{321}$ is determined by its zero set $Z(p)$. Its weight is
\[
w(p)=\sum_{i \in \positive{p}} p_i = \sum_{i \in \positive{p}} |\{z\in Z(p): i < z \leq n \}|.
\]
Equivalently, $w(p)$ counts pairs $(i,z)$ such that $i<z$ where $i\in \positive{p}$ and $z\in Z(p)$.

Let $|Z(p)|=k$. We have $w(p)\leq k(n-k)$, with equality holding exactly when every nonzero position lies to the left of every zero position. In this case the corresponding maximal code is
\[
p=(\underbrace{k,\ldots,k}_{n-k}, \underbrace{0,\ldots,0}_{k}),
\]
Therefore
\[
\max_{p \in \lehmerpat{n}{321}} w(p) =\max_{1\leq k\leq n-1} k(n-k) =\left\lfloor \frac{n^2}{4}\right\rfloor.
\]
If $n=2r$ is even, this maximum is attained uniquely at $k=r$, and if $n=2r+1$, then this maximum is attained at both $k=r$ and $k=r+1$.
\end{proof}

\section{Lehmer Codes for Consecutive \underline{321}-Avoiding Permutations}

\label{sec:consec-321}

In this section, we turn to consecutive $\underline{321}$-avoiding permutations. 
We first translate $\underline{321}$-avoidance into a condition on descents in the Lehmer code.  We then show that every maximal code $p \in \maxset{n}{\underline{321}}$ is determined by its descent set $D(p)$. This  leads to enumeration of these maximal codes by the Padovan sequence. 

%%%%%%%%%%%%%%%Lehmer Condition%%%%%%%%%%%%%%%

We begin with a basic observation relating permutation descents to Lehmer code descents. 
\begin{lemma}
\label{lem:adjacent-decrease}
    Let $\pi\in S_n$ have Lehmer code $p=(p_1,\ldots,p_n)$. Then for every $1\le i\le n-1$, we have
    \[
    \pi_i>\pi_{i+1}
    \qquad \mbox{if and only if}\qquad
    p_i>p_{i+1}.
    \]
    Equivalently, the descents of $\pi$ correspond to the descents of its Lehmer code $p$.
\end{lemma}

\begin{proof}
    Recall that $p_i$ counts the number of entries to the right of $\pi_i$ that are smaller than $\pi_i$.
    
    First suppose that $\pi_i>\pi_{i+1}$. Every entry to the right of $\pi_{i+1}$ that is smaller than $\pi_{i+1}$ is also smaller than $\pi_i$. In addition, $\pi_{i+1}$ itself is to the right of $\pi_i$ and is smaller than $\pi_i$. Hence, $p_i\ge p_{i+1}+1,$ so $p_i>p_{i+1}$.
    
    Conversely, suppose that $\pi_i<\pi_{i+1}$. Every entry to the right of $\pi_i$ that is smaller than $\pi_i$ lies to the right of $\pi_{i+1}$ and is also smaller than $\pi_{i+1}$. Therefore, $p_i\le p_{i+1}$.
\end{proof}

Using this observation twice, an occurrence of the consecutive pattern $\underline{321}$ in permutation $\pi$ corresponds to three consecutive entries of the Lehmer code $p=L(\pi)$ in decreasing order.

\begin{theorem}
\label{thm:-321-avoiding-lehmer}
A permutation $\pi \in \perm{n}$ is $\underline{321}$-avoiding if and only if its Lehmer code $p=(p_1,\ldots,p_n)\in \lehmer{n}$ satisfies the following condition:
\begin{equation}
\label{eqn:-321-condition}
\mbox{if } p_i>p_{i+1} \mbox{ then } p_{i+1}\leq p_{i+2} \mbox{ for } 1 \leq i \leq n-2.
\end{equation}
\end{theorem}

\begin{proof}
The permutation $\pi$ contains an occurrence of $\underline{321}$ starting at position $i$ exactly when $\pi_i>\pi_{i+1}>\pi_{i+2}$. By Lemma \ref{lem:adjacent-decrease} (twice), this is equivalent to $p_i>p_{i+1}>p_{i+2}$. Therefore $\pi$ avoids $\underline{321}$ if and only if condition \eqref{eqn:-321-condition} holds.
\end{proof}

One immediate consequence is that the descent set of a $\underline{321}$-avoiding code has no consecutive indices.

\begin{cor}
\label{cor:descent-increase-by-2}
Let $p\in\lehmer{n}$ with descent set $D(p)=\{d_1,\ldots,d_m\}$. Then $p\in\lehmerpat{n}{\underline{321}}$ if and only if $d_{i+1}-d_i\ge 2$ for $1\le i\le m-1$.
\end{cor}

\begin{proof}
By Theorem \ref{thm:-321-avoiding-lehmer}, $p\in\lehmerpat{n}{\underline{321}}$ if and only if there is no
$i\in[n-2]$ such that $p_i>p_{i+1}>p_{i+2}$. This is equivalent to $p$ not having descents at both $i$ and $i+1$, which holds if and only if $d_{i+1}-d_i\ge 2$ for $1\le i\le m-1$.
\end{proof}

%%%%%%%%%%%%%%%Maximal Condition%%%%%%%%%%%%%%%

We now characterize the maximal codes $p \in \Max \lehmerpat{n}{\underline{321}}$. Our next lemma shows that a maximal $p$ is weakly decreasing. Furthermore,  each descent $d \in D(p)$ must attain its largest possible Lehmer code value $p_d = n-d$. We also want to have as many descents as possible, subject to the non-adjacency condition. 

\begin{lemma}
\label{lem:maximal_condition_consec_321}
For $n \geq 2$, let $p=(p_1,\ldots,p_n)\in \lehmer{n}$ with $D(p)=\{d_1,d_2, \ldots, d_m\}$. Then $p \in \Max \lehmerpat{n}{\underline{321}}$ if and only if the following conditions hold.
\begin{enumerate}
    \item[(a)] We have
\begin{equation}
\label{eqn:desc_set_consec_321}
          d_1\in\{1,2\},\qquad d_{j+1}-d_j\in\{2,3\}\text{ for }1\le j \le m-1, \qquad d_m\in\{n-2,n-1\}.  
\end{equation}

    \item[(b)] For $1\le j\le m$, $p_{d_j}=n-d_j$.
    
    \item[(c)] The code $p$ is constant on the subintervals
    \[
        [1,d_1],\quad [d_1+1,d_2],\quad [d_2+1,d_3],\quad \ldots,\quad [d_{m-1}+1,d_m],\quad [d_m+1,n],
    \]
    which we call the blocks of $p$.
\end{enumerate}
\end{lemma}

Figure \ref{fig:maxconsec321} shows the codes in $\Max \lehmerpat{7}{\underline{321}}$ with their descent sets. 
Note that condition (a) is equivalent to saying that $D(p)$ is a maximal subset of $[n-1]$ that contains no consecutive elements.
After the descent set is chosen, conditions (b) and (c) determine the code itself.

\begin{figure}[t]
\begin{center}
\begin{tikzpicture}[scale=.6]

    \newtri[(0,0)]{5,5,2,2,2,0}{teal};
    \node at (0,-0.8) {\footnotesize $\{2,5\}$};

    \newtri[(3.5,0)]{5,5,3,3,1,1}{teal};
    \node at (3.5,-0.8) {\footnotesize $\{2,4,6\}$};

    \newtri[(7,0)]{6,3,3,3,1,1}{teal};
    \node at (7,-0.8) {\footnotesize $\{1,4,6\}$};

    \newtri[(10.5,0)]{6,4,4,1,1,1}{teal};
    \node at (10.5,-0.8) {\footnotesize $\{1,3,6\}$};

    \newtri[(14,0)]{6,4,4,2,2,0}{teal};
    \node at (14,-0.8) {\footnotesize $\{1,3,5\}$};
  
    \node at (-2, 0.5) {\footnotesize $p$};
    \node at (-2, -0.8) {\footnotesize $D(p)$};

\end{tikzpicture}
\end{center}
\caption{The triangle representations of the maximal Lehmer codes in $\maxset{7}{\underline{321}}$, labeled by their descent sets.}
\label{fig:maxconsec321}
\end{figure}
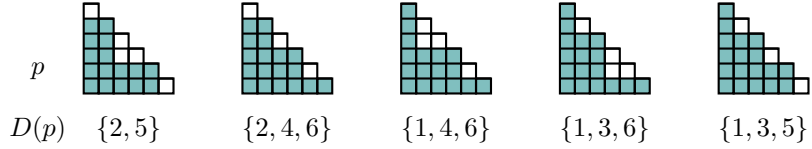

\begin{proof}
By Corollary \ref{cor:descent-increase-by-2}, code $p \in \lehmerpat{n}{\underline{321}}$ if and only if $D(p)$ has no consecutive indices.

First suppose that $p$ is maximal. If $p_i<p_{i+1}$ for some $i$, then replacing $p_i$ by $p_{i+1}$ gives a larger valid Lehmer code and creates no pair of consecutive descents,
contradicting the maximality of $p$. Hence $p$ is weakly decreasing, so it is constant on the block intervals
between consecutive descents, proving (c).

Next let $d_j\in D(p)$. If $p_{d_j}<n-d_j$, then increasing $p_{d_j}$ by $1$ still gives a
valid Lehmer code. Since $D(p)$ has no consecutive indices, this creates no pair of consecutive descents, again contradicting the maximality of $p$. Thus $p_{d_j}=n-d_j$ for every $j$, proving (b).

Finally, we claim that (a) holds, so that  $D(p)$ is  inclusion-maximal among subsets of $[n-1]$ with no consecutive indices. We prove the contrapositive: if $p \in \lehmerpat{n}{\underline{321}}$ does not satisfy condition (a), then there exists $q \in \lehmerpat{n}{\underline{321}}$ such that $q \succ p$. There are three cases to consider.
First, if $d_1 \geq 3$, then create $q$ by increasing $p_1$ from $n-d_1$ to $n-1$, which adds a descent at index 1. 
Second, if $d_m \leq n-3$, then create $q$ by increasing the values on $[d_m+1, n-1]$ from 0 to 1, which adds a descent at index $n-1$.
Third, if $d_{j+1}-d_{j} \geq 4$ for some $1 \leq j \leq m-1$, then create $q$ by increasing the values on $[d_{j}+1,d_{j}+2]$ from $n-d_{j+1}$ to $n-d_j-2$, which adds a descent at $d_j+2$.

Conversely, suppose that (a), (b), and (c) hold. Then $D(p)$ has no consecutive indices, so
$p\in\lehmerpat{n}{\underline{321}}$. Let $q\in\lehmerpat{n}{\underline{321}}$ satisfy $p\preceq q$. Since $p_{d_j}=n-d_j$ is already the largest possible value at position $d_j$,
we have $q_{d_j}=p_{d_j}$ for every $j$. Hence $D(p) \subseteq D(q)$.
If $q\neq p$, then by the block-constant form in (c), some block contains an entry of $q$ strictly larger than the corresponding entry of $p$. Moving right within that block, we encounter a descent of $q$ before reaching the end of the block, so that $D(p) \subsetneq D(q)$. But by (a), $D(p)$ is inclusion-maximal among subsets of $[n-1]$ with no consecutive indices,
so $D(q)$ must contain two consecutive indices. This contradicts $q\in\lehmerpat{n}{\underline{321}}$. Therefore $q=p$, and $p$ is maximal.
\end{proof}

Lemma \ref{lem:maximal_condition_consec_321} shows that a maximal code $p \in \Max \lehmerpat{n}{\underline{321}}$ is completely determined by its descent set. For example, we construct the code $p \in \Max \lehmerpat{8}{\underline{321}}$ with descent set $D=\{1,3,6\}$. First, we note that this set adheres to condition (a). We then use conditions (b) and (c) to obtain $p=(7,5,5,2,2,2,0,0)$. The corresponding $\underline{321}$-avoiding permutation is $\pi=86734512$.

%%%%%%%%%%%%%%%Maximal Size%%%%%%%%%%%%%%%%%%%%%

The enumeration of maximal codes $\Max \lehmerpat{n}{\underline{321}}$ is now reduced to counting the possible descent sets. Condition \eqref{eqn:desc_set_consec_321} in Lemma \ref{lem:maximal_condition_consec_321} gives a simple recurrence, according to whether the first descent is at position $1$ or position $2$. The enumeration 
matches the Padovan numbers, sequence \oeis{A000931} in the OEIS \cite{oeis}, shifted by one index.

\begin{theorem*}[Theorem \ref{thm:consec_321_number}]
The number of maximal $\underline{321}$-avoiding Lehmer codes is given by the Padovan sequence $|\maxset{n}{\underline{321}}| = P_{n-1}$, where $P_0=P_1=1$ and $P_2=2$, and $P_n = P_{n-2} + P_{n-3}$ for $n \geq 3$.
\end{theorem*}

\begin{proof}
For $n \geq 2$, let $\mathcal{D}_n = \{ D(p) : p \in \Max \lehmerpat{n}{\underline{321}}\}$ and define $\mathcal{D}_1 = \{ \emptyset \}$. 
By Lemma \ref{lem:maximal_condition_consec_321}, we have $|\mathcal{D}_n|= |\Max \lehmerpat{n}{\underline{321}}|$. Furthermore, for $n \geq 2$, $\mathcal{D}_n$ is the collection of subsets of $[n-1]$ satisfying condition \eqref{eqn:desc_set_consec_321}.
We have $\mathcal{D}_2 = \{ \{1\} \}$ and $\mathcal{D}_3 = \{ \{1\}, \{2\} \}$, so $|\mathcal{D}_1|=1$, $|\mathcal{D}_2|=1$ and $|\mathcal{D}_3|=2$.

For $n\ge 4$, partition $D\in\mathcal{D}_n$ according to its first descent. If $d_1=1$, then the remaining descents lie in $\{3,4,\ldots,n-1\}$, and after shifting all remaining descents down by $2$, we obtain an element of $\mathcal{D}_{n-2}$. Conversely, every element of $\mathcal{D}_{n-2}$ gives such a $D$ by shifting up by $2$ and adding element $1$ to the set.
If $d_1=2$, then the remaining descents lie in $\{4,5,\ldots,n-1\}$, and after shifting all remaining descents down by $3$, we obtain an element of $\mathcal{D}_{n-3}$. Conversely, every element of $\mathcal{D}_{n-3}$ gives such a $D$ by shifting up by $3$ and adding element $2$ to the set. Therefore
\[
|\mathcal{D}_n|=|\mathcal{D}_{n-2}|+|\mathcal{D}_{n-3}|\qquad\text{for }n\ge 4.
\]
Thus $|\maxset{n}{\underline{321}}|=|\mathcal{D}_n| = P_{n-1}$ is the shifted Padovan sequence.
\end{proof}

%%%%%%%%%%%%%%%Maximum Weight%%%%%%%%%%%%%%%%%%%%%

Finally, we use the block structure from Lemma \ref{lem:maximal_condition_consec_321} to determine the maximum weight. It is convenient to compare a maximal code with the full Lehmer code $(n-1,n-2,\ldots,1,0)$ and minimize the resulting defect.

\begin{theorem*}[Theorem \ref{thm:maximum-weight-321}]
For $n \geq 2$, we have
$$
\max \{ \inv(\pi) : \pi \in \permpat{n}{\underline{321}}  \} = 
\max \{ w(p) : p \in \lehmerpat{n}{\underline{321}} \}
=
\left\lceil \frac{(n-1)^2}{2}\right\rceil.
$$
Moreover, if $n = 2r$ is even, the unique maximum weight element is
\[
(2r-1, 2r-3, 2r-3, 2r-5, 2r-5, \ldots, 3, 3, 1, 1, 0)
\]
and when $n=2r+1$ is odd, the maximum weight elements are
\[
(2r, 2r-2, 2r-2, 2r-4, 2r-4, \ldots, 4, 4, 2, 2, 0, 0)
\]
and
\[
(2r-1, 2r-1, 2r-3, 2r-3, \ldots, 3, 3, 1, 1,  0).
\]
\end{theorem*}

\begin{figure}[t]
\begin{center}
\begin{tikzpicture}[scale=.6]

    % n = 2
    \newtri[(0,0)]{1}{teal};
    \node at (0,-0.8) {\footnotesize $(1,0)$};
    \node at (0,-1.55) {\footnotesize $n=2$};

    % n = 3
    \newtri[(3.2,0)]{2,0}{teal};
    \node at (3.2,-0.8) {\footnotesize $(2,0,0)$};

    \newtri[(5.8,0)]{1,1}{teal};
    \node at (5.8,-0.8) {\footnotesize $(1,1,0)$};

    \node at (4.5,-1.55) {\footnotesize $n=3$};

    % n = 4
    \newtri[(9.2,0)]{3,1,1}{teal};
    \node at (9.2,-0.8) {\footnotesize $(3,1,1,0)$};
    \node at (9.2,-1.55) {\footnotesize $n=4$};

    % n = 5
    \newtri[(13.0,0)]{4,2,2,0}{teal};
    \node at (13.0,-0.8) {\footnotesize $(4,2,2,0,0)$};

    \newtri[(16.2,0)]{3,3,1,1}{teal};
    \node at (16.2,-0.8) {\footnotesize $(3,3,1,1,0)$};

    \node at (14.6,-1.55) {\footnotesize $n=5$};

    % n = 6
    \newtri[(20.4,0)]{5,3,3,1,1}{teal};
    \node at (20.4,-0.8) {\footnotesize $(5,3,3,1,1,0)$};
    \node at (20.4,-1.55) {\footnotesize $n=6$};

\end{tikzpicture}
\end{center}
\caption{The triangle representations of the maximum weight $\underline{321}$-avoiding Lehmer codes for $2\leq n\leq 6$. For odd $n$, there are two maximum weight elements. For even $n$, the maximum weight element is unique.}
\label{fig:maxweightconsec321examples}
\end{figure}
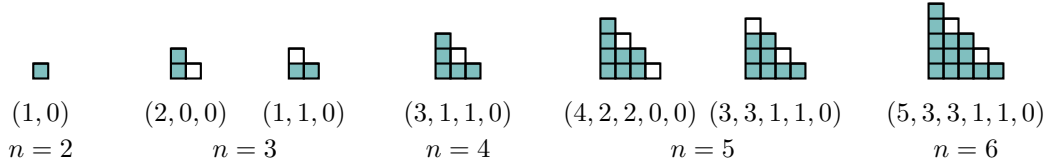

Figure \ref{fig:maxweightconsec321examples} shows the maximum weight $\underline{321}$-avoiding Lehmer codes for $2 \leq n \leq 6$. 

\begin{proof}
The first equality holds because when $\pi$ has Lehmer code $p$, we have $\inv(\pi)=w(p)$, as noted in Definition \ref{def:lehmer}.

Let $p=(p_1,\ldots,p_n)\in \lehmerpat{n}{\underline{321}}$ be a maximum weight code. Since every maximum weight code is maximal, Lemma \ref{lem:maximal_condition_consec_321} implies that $p$ is determined by its descent set
$D(p)=\{d_1,d_2,\ldots,d_m\}$, which satisfies condition \eqref{eqn:desc_set_consec_321}.
Setting $d_0=0$ and $d_{m+1}=n$, we have
\[
p_i=n-d_j \qquad \text{whenever } i \in [d_{j-1}+1, d_j], \quad 1\le j\le m+1.
\]
We compare $p$ with the full Lehmer code
    $(n-1,n-2,\ldots,1,0)$, defining the \emph{defect} of $p$ by
\[
    \delta(p)=\sum_{i=1}^n (n-i-p_i).
\]
Since
\[
    w(p)=\sum_{i=1}^n p_i=\binom n2-\delta(p),
\]
maximizing $w(p)$ is equivalent to minimizing $\delta(p)$.

Let $(b_1, \ldots, b_{m+1})$ be the \emph{block length sequence} of $p$, where
\[
    b_j=d_j-d_{j-1}\text{ for }1\le j\le m+1,
\]
where we set $d_0=0$ and $d_{m+1}=n$. Then $\sum_{i=1}^{m+1} b_i = n$, where
\[
    b_1\in\{1,2\},\qquad
    b_j\in\{2,3\}\text{ for }2\le j\le m,\qquad
    b_{m+1}\in\{1,2\}.
\]
On block $[d_{j-1}+1, d_j]$ of length $b_j$, the entries of $p$ are constant and equal to $n-d_j$. Hence the defect contributed by this block is
\[
    0+1+\cdots+(b_j-1)=\binom{b_j}{2}.
\]
Therefore
\[
    \delta(p)=\sum_{j=1}^{m+1}\binom{b_j}{2}.
\]

For the allowed block lengths, the contributions to the defect are
\[
    \binom{1}{2}=0,\qquad\binom{2}{2}=1,\qquad\binom{3}{2}=3,
\]
and these values are independent of the location of the block. So we are free to reorder the blocks without changing the defect while still adhering to condition \eqref{eqn:desc_set_consec_321}, so that $1$-blocks can only appear as either the first or last block, and $3$-blocks cannot appear as either the first or last block. 
So let us reorder the blocks so that the $3$-blocks appear as far to the left as possible.

We claim that a block length sequence that minimizes the defect does not contain any $3$-blocks. Indeed, 
replacing two $3$-blocks with three $2$-blocks lowers the defect. So there is at most one $3$-block.
If the sequence starts with a $1$-block and a $3$-block, then replacing them with two $2$-blocks lowers the defect.
If the sequence starts with a $2$-block and a $3$-block, then replacing them with one $1$-block and two $2$-blocks lowers the defect.

If $n=2r$ is even, having a $1$-block at each end gives a smaller defect than having all $2$-blocks. So 
the unique minimum-defect block length sequence is
$
    (1,2,2,\ldots,2,1).
$
This corresponds to
$
    D(p)=\{1,3,5,\ldots,2r-1\}.
$
There are $r-1$ blocks of size $2$, so $\delta(p)=r-1$. Hence
\[
    w(p)=\binom{2r}{2}-(r-1) =2r^2-2r+1 =\left\lceil \frac{(2r-1)^2}{2}\right\rceil =\left\lceil \frac{(n-1)^2}{2}\right\rceil.
\]
This descent set corresponds to the unique even length code in the theorem statement.

If $n=2r+1$ is odd, then the two minimum-defect block length sequences are
$
    (1,2,2,\ldots,2,2) 
    \quad \mbox{and} \quad 
    (2,2,\ldots,2,1).
$
These sequences correspond respectively to
$D(p)=\{1,3,5,\ldots,2r-1\}$ and
$D(p)=\{2,4,6,\ldots,2r\}$.
In either case, there are $r$ blocks of size $2$, so $\delta(p)=r$. Therefore
\[
    w(p)=\binom{2r+1}{2}-r = 2r^2=\left\lceil \frac{(2r)^2}{2}\right\rceil=\left\lceil \frac{(n-1)^2}{2}\right\rceil.
\]
These two descent sets correspond to the two odd length codes in the theorem statement.
\end{proof}

\bibliographystyle{plain}
\bibliography{biblio}

@article{Babson2000, 
title={Generalized permutation patterns and a classification of the {M}ahonian statistics}, 
author={Babson, Eric and Steingr\'imsson, Einar}, 
volume={44}, 
journal={S\'eminaire Lotharingien de Combinatoire}, 
year={2000},  
pages={Article B44b} 
}

@misc{BHR,
      title={$\underline{32}1$-Avoiding Permutations with Maximum Inversion Number}, 
      author={Andrew Beveridge and Kristin Heysse and Paige Robertson},
      howpublished={arXiv preprint arXiv:2607.18417 [math.CO]},
      year={2026},
      note={Available at \url{https://arxiv.org/abs/2607.18417}}
}

@misc{BHL2026,
    title = {Lehmer Codes and the Reverse-Complement Mapping from \underline{32}1-Avoiding Permutations to 3\underline{21}-Avoiding Permutations},
    author = {Andrew Beveridge and Yufan Hu and Yucheng Liu},
    howpublished={arXiv preprint arXiv:2607.26900 [math.CO]},
    year={2026},
    note = {Available at \url{https://arxiv.org/abs/2607.26900}}
}

@article{bouvel,
	title = {Between Weak and {Bruhat}: the Middle Order on Permutations},
	volume = {41},
	issn = {1435-5914},
	url = {https://doi.org/10.1007/s00373-024-02885-3},
	doi = {10.1007/s00373-024-02885-3},
	number = {2},
	journal = {Graphs and Combinatorics},
	author = {Bouvel, Mathilde and Ferrari, Luca and Tenner, Bridget Eileen},
	month = feb,
	year = {2025},
	pages = {34},
}

@Article{denoncourt,
  Title                    = {A refinement of weak order intervals into distributive lattices},
  Author                   = {Hugh Denoncourt},
  Journal                  = {Annals of Combinatorics},
  Year                     = {2013},
    volume		={17},
  Pages                    = {655-670},

}

@incollection{Elizalde2016,
author="Elizalde, Sergi",
editor="Beveridge, Andrew
and Griggs, Jerrold R.
and Hogben, Leslie
and Musiker, Gregg
and Tetali, Prasad",
title="A survey of consecutive patterns in permutations",
bookTitle="Recent Trends in Combinatorics",
year="2016",
publisher="Springer International Publishing",
pages="601--618",
isbn="978-3-319-24298-9",
doi="10.1007/978-3-319-24298-9_24",
url="https://doi.org/10.1007/978-3-319-24298-9_24"
}

@article{ElizaldeNoy2003,
  author  = {Sergi Elizalde and Marc Noy},
  title   = {Consecutive patterns in permutations},
  journal = {Advances in Applied Mathematics},
  volume  = {30},
  number  = {1--2},
  pages   = {110--125},
  year    = {2003},
  doi     = {10.1016/S0196-8858(02)00527-4}
}

@article{Frosini2025,
  author  = {Frosini, Andrea and Guerrini, Veronica and Rinaldi, Simone},
  title   = {Constrained Underdiagonal Paths and Pattern Avoiding Permutations},
  journal = {Mathematics},
  volume  = {13},
  number  = {3},
  pages   = {517},
  year    = {2025},
  doi     = {10.3390/math13030517}
}

@book{kitaev,
  title     = "Patterns in Permutations and Words",
  author    = "Kitaev, Sergey",
  year      = 2011,
  publisher = "Springer-Verlag"
}

@book{knuth,
  title     = "The Art of Computer Programming, Volume 1: Fundamental Algorithms",
  author    = "Knuth, Donald E.",
  year      = 1968,
  publisher = "Addison-Wesley"
}

@inproceedings{lehmer,
  author       = {Derrick H. Lehmer},
  editor       = {Bellman, R. and Hall, Jr., M.},
  title        = {Teaching combinatorial tricks to a computer},
  year         = {1960},
  series       = {Proceedings of Symposia in Applied Mathematics},
  booktitle    = {Combinatorial Analysis},
  publisher={American Mathematical Society},
    volume                   = {10},
  pages        = {179--193},
}

@article{mansour2022, 
title={On a question of {L}i concerning
an uncounted class of circular permutations}, 
author={Mansour, Toufik and Shattuck, Mark}, 
volume={83},
number={1},
journal={Australasian Journal of Combinatorics}, 
year={2022},  
pages={176–195} 
}

@misc{oeis,
author={Neil J. A. Sloane and The OEIS Foundation Inc.},
title={The {O}n-{L}ine {E}ncyclopedia of {I}nteger {S}equences},
howpublished={\url{https://oeis.org}},
note={2026}
}

@article{simion1985restricted,
  title={Restricted permutations},
  author={Simion, Rodica and Schmidt, Frank W.},
  journal={European Journal of Combinatorics},
  volume={6},
  number={4},
  pages={383--406},
  year={1985},
  publisher={Academic Press},
  doi={10.1016/S0195-6698(85)80052-4}
}

@incollection{steingrimsson, 
place={Cambridge}, 
series={London Mathematical Society Lecture Note Series}, 
title={Generalized permutation patterns – a short survey}, 
booktitle={Permutation Patterns}, 
publisher={Cambridge University Press}, 
author={Steingrímsson, Einar}, 
editor={Linton, Steve and Ruškuc, Nik and Vatter, Vincent}, 
year={2010}, 
pages={137--152}, 
collection={London Mathematical Society Lecture Note Series}
}

@Article{tomie,
  Title                    = {Two characterizations of the shape of the base poset derived from the {L}ehmer code of a permutation using permutation patterns},
  Author                   = {Masaya Tomie},
  Journal                  = {Journal of Combinatorics},
  Year                     = {2014},
  volume		={5},
  number		={4},
  Pages                    = {499-514},

}

\end{document}